\documentclass[a4paper, 10pt, twoside, notitlepage]{amsart}

\usepackage[utf8]{inputenc}
\usepackage{color}
\usepackage{amsmath} 
\usepackage{amssymb} 
\usepackage{amsthm}
\usepackage{geometry}
\usepackage{graphicx}
\usepackage{mathtools}
\usepackage{esint}
\usepackage[colorlinks=true,linkcolor=blue]{hyperref}

\theoremstyle{plain}
\newtheorem{thm}{Theorem}
\newtheorem{prop}{Proposition}[section]
\newtheorem{lem}[prop]{Lemma}

\newtheorem{rmk}[prop]{Remark}

\newcommand {\R} {\mathbb{R}}

\newcommand {\p} {\partial}
\newcommand {\pn} {\partial_{\nu}}

\newcommand {\D} {\Delta}

\DeclareMathOperator {\dist} {dist}

\title{Quantitative Runge Approximation and Inverse Problems}

\author[A. R\"uland]{Angkana R\"uland}
\address{Mathematical Institute, University of Oxford, Andrew Wiles Building, Radcliffe Observatory Quarter, Woodstock Road, Oxford OX2 6GG}
\email{ruland@maths.ox.ac.uk}

\author[M. Salo]{Mikko Salo}
\address{Department of Mathematics and Statistics, University of Jyv\"askyl\"a}
\email{mikko.j.salo@jyu.fi}

\begin{document}
\maketitle

\begin{abstract}
In this short note we provide a quantitative version of the classical Runge approximation property for second order elliptic operators. This relies on quantitative unique continuation results and duality arguments. We show that these estimates are essentially optimal. As a model application we provide a new proof of the result from \cite{F07}, \cite{AK12} on stability for the Calder\'on problem with local data. 
\end{abstract}

\section{Introduction}
\label{sec:intro}

In this note we study \emph{quantitative Runge type approximation properties} for elliptic equations. 
If $D_1,D_2$ are bounded, open Lipschitz sets (e.g.\ balls) with $D_1\Subset D_2$, a variant of the \emph{classical Runge approximation property} for uniformly elliptic equations states that it is possible to approximate (for instance with respect to the strong $L^2$ topology) solutions to an equation in the smaller domain by solutions of the same equation in the larger domain. 

Let us formulate this more precisely: Let $D_2$ be a bounded Lipschitz domain in $\R^n$, $n \geq 2$, and consider the operator
\begin{align}
\label{eq:L}
L := \p_i a^{ij} \p_j  + c,
\end{align}
where $(a^{ij}) \in L^{\infty}(D_2, \R^{n\times n})$ is a symmetric matrix function, $c \in L^{\infty}(D_2)$, and for some $K \geq 1$ 
\begin{equation} \label{l_bounds}
\begin{gathered}
K^{-1} |\xi|^2 \leq a^{ij}(x) \xi_i \xi_j \leq K |\xi|^2 \mbox{ for a.e.\ $x \in D_2$ and for all $\xi \in \R^n$}, \\
 \| c \|_{L^{\infty}(D_2)} \leq K, \\
\mbox{and if $n \geq 3$ then also } \| \nabla a^{ij} \|_{L^{\infty}(D_2)} \leq K.
\end{gathered}
\end{equation}
Here we use the summation convention, and for simplicity we do not consider drift terms. We note that
$L$ is formally self-adjoint: $L^{\ast}:= \p_i a^{ij} \p_j + c$. We also make the standing assumption that $0$ is not a Dirichlet eigenvalue for $L$ in $D_2$, so that the Dirichlet problem for $L$ (and $L^*$) in $D_2$ is well-posed.

Further, let $D_1$ be another bounded Lipschitz domain in $\R^n$ so that $D_1 \Subset  D_2$ and $D_2 \setminus \overline{D}_1$ is connected, and let $\Gamma$ be a nonempty open subset of $\partial D_2$. We will compare solutions to the homogeneous equation associated with \eqref{eq:L} on the domains $D_1$ and $D_2$, where additionally the boundary value of the solution on $D_2$ vanishes outside $\Gamma$. We introduce the following spaces:
\begin{align*}
S_1&:=\{h\in H^1(D_1): L w = 0 \mbox{ in } D_1\},\\
S_2&:=\{h\in H^1(D_2): L w = 0 \mbox{ in } D_2, \ w|_{\partial D_2} \in \widetilde{H}^{1/2}(\Gamma) \}.
\end{align*}
The notation $w|_{\partial D_2} \in \widetilde{H}^{1/2}(\Gamma)$ is defined in Section \ref{sec:qual}, and it implies that $\mathrm{supp}(w|_{\partial D_2}) \subset \overline{\Gamma}$. In this notation, the \emph{classical Runge approximation} due to Lax \cite{L56} and Malgrange \cite{M54} asserts the following density property:

\begin{thm}[\cite{L56}, \cite{M54}]
\label{thm:CR}
Let $L$ be the operator from \eqref{eq:L} and let $D_1,D_2,\Gamma$ and $S_1,S_2$ be as above. Then for any $\epsilon>0$ and any $h \in S_1$ there exists $u\in S_2$ such that
\begin{align*}
\|h-u|_{D_1}\|_{L^2(D_1)} \leq \epsilon.
\end{align*}
\end{thm}

A prototypical example of such an approximation result is given by harmonic functions, which, by analyticity, can always be approximated by harmonic polynomials. Extensions of Theorem \ref{thm:CR} to a much more general class of operators were considered for instance in \cite{B62}, \cite{B62a}.

As the main result of this note, we derive a \emph{quantitative} version of Theorem \ref{thm:CR}: Given an  error threshold $\epsilon > 0$ and a solution $h \in S_1$, we estimate the size of a solution in $S_2$ that approximates $h$ in $D_1$ up to error $\epsilon$. 

\begin{thm}
\label{thm:Runge_quant}
Let $L$ be the operator from \eqref{eq:L} and let $D_1,D_2,\Gamma$ and  $S_1,S_2$ be as above.
There exist a parameter $\mu>0$ and a constant $C>1$ (depending on $D_1, D_2, \Gamma, n$, $K$) such that for each function $h \in S_1$ and each error threshold $\epsilon\in(0,1)$, there exists a function $u \in S_2$ with
\begin{align}
\label{eq:approx}
\|h-u|_{D_1}\|_{L^2(D_1)} \leq \epsilon \|h\|_{H^{1}(D_1)}, \ 
\|u\|_{H^{1/2}(\partial D_2)} \leq C e^{C \epsilon^{-\mu}} \|h\|_{L^2(D_1)}.
\end{align}
\end{thm}

We remark that up to the precise power $\mu>0$ the exponential bound in $\epsilon$ is optimal, which can be seen by considering spherical harmonics (see Section \ref{sec:optimal}). However, if $h$ is assumed to be a solution in a slightly larger domain, one obtains polynomial bounds instead:

\begin{thm}
\label{thm:Runge_quant2}
Let $L$, $D_1$, $D_2$, $\Gamma$, and $S_2$ be as above. Let also $\tilde{D}$ be a bounded Lipschitz domain with $D_1 \Subset \tilde{D} \Subset D_2$. There exist $C, \mu \geq 1$ (depending on $D_1$, $D_2$, $\tilde{D}$, $\Gamma$, $n$, $K$) such that for each $\tilde{h} \in H^1(\tilde{D})$ with $L\tilde{h} = 0$ in $\tilde{D}$ and for each $\epsilon\in(0,1)$, there exists $u \in S_2$ with
\begin{align}
\label{eq:approx2}
\|\tilde{h}|_{D_1}-u|_{D_1}\|_{L^2(D_1)} \leq \epsilon \|\tilde{h}\|_{H^{1}(\tilde{D})}, \ 
\|u\|_{H^{1/2}(\partial D_2)} \leq C \epsilon^{-\mu} \|\tilde{h}|_{D_1}\|_{L^2(D_1)}.
\end{align}
\end{thm}

As in \cite{RS17}, which studied quantitative approximation properties for nonlocal equations, the argument for Theorems \ref{thm:Runge_quant} and \ref{thm:Runge_quant2} relies on a quantitative unique continuation result for a ``dual equation", see Proposition \ref{prop:approx_qwuc}, combined with a functional analysis argument which we borrow from control theory \cite{R95}.

This work is partly motivated by applications to inverse problems, where qualitative Runge type approximation results have successfully been applied in various contexts. Focusing particularly on Calder\'on type problems, we mention \cite{KV85, I88} related to determination of piecewise analytic or discontinuous conductivities, the probe method \cite{Ikehata1998, Ikehata2013} and oscillating-decaying solutions \cite{NUW05} for inclusion detection, local data results when the conductivity is known near the boundary \cite{AU04}, and monotonicity based methods \cite{G08} involving localized potentials that are closely related to Runge approximation.

We will use Theorem \ref{thm:Runge_quant2} to provide a new proof of the result from \cite{F07}, \cite{AK12} on the stability of the result in \cite{AU04} for the Calder\'on problem with local data. While this stability result itself is not new, we view the problem as a model setting that demonstrates the strength of quantitative Runge approximation in connection with inverse problems. We hope that the method might be useful in other settings as well.

\subsection*{Organization of the article}
The remainder of the note is organized as follows: Section \ref{sec:qual} recalls the argument for qualitative Runge approximation (Theorem \ref{thm:CR}). In Section \ref{sec:quant} we discuss a quantitative unique continuation principle, which is proved as a consequence of results in \cite{ARV09}. Section \ref{sec:quantrunge} proves the quantitative approximation results (Theorems \ref{thm:Runge_quant} and \ref{thm:Runge_quant2}) by reducing them to quantitative unique continuation via a functional analysis argument. In Section \ref{sec:optimal} we consider a spherically symmetric set-up proving the optimality of the bounds in \eqref{eq:approx}. Finally, in Section \ref{sec:stability}, as a model application of the quantitative Runge approximation, we discuss stability for the Calder\'on problem with local data.

\subsection*{Acknowledgements}
A.R.\ gratefully acknowledges a Junior Research Fellowship at Christ Church. M.S.\ was supported by the Academy of Finland (Finnish Centre of Excellence in Inverse Problems Research, grant number 284715) and an ERC Starting Grant (grant number 307023).

\section{Qualitative Runge Approximation}
\label{sec:qual}

In this section, we recall for completeness the proof of Theorem \ref{thm:CR}. For the proof, we also introduce some notation and facts that will be useful later. We begin with a standard lemma concerning weak solutions in Lipschitz domains. Here and below, we write $(\,\cdot\,,\,\cdot\,)_{L^2}$ both for the $L^2$ inner product and for the distributional pairing between a distribution and a test function.

\begin{lem} \label{lemma_solvability_basic}
Let $\Omega \subset \R^n$, $n \geq 2$, be a bounded open set with Lipschitz boundary. Let $L$ be of the form \eqref{eq:L} where $(a^{ij})$ is symmetric, and for some $M \geq 1$ one has $M^{-1} |\xi|^2 \leq a^{ij} \xi_i \xi_j \leq M |\xi|^2 $ a.e.\ on $\Omega$ and $\| c \|_{L^{\infty}(\Omega)} \leq M$. Assume that $0$ is not a Dirichlet eigenvalue of $L$ in $\Omega$.

Then for any $F \in L^2(\Omega)$ and $g \in H^{1/2}(\partial \Omega)$, the problem 
\[
Lu = F \mbox{ in } \Omega, \quad u = g \mbox{ on } \partial \Omega
\]
has a unique solution $u \in H^1(\Omega)$ satisfying 
\[
\| u \|_{H^1(\Omega)} \leq C ( \| F \|_{L^2(\Omega)} + \| g \|_{H^{1/2}(\partial \Omega}).
\]
The conormal derivative $\pn u \in H^{-1/2}(\partial \Omega)$ is defined in the weak sense via 
\[
(\pn u, \tilde{g})_{L^2(\partial \Omega)} := (a^{ij} \partial_i u, \partial_j E \tilde{g})_{L^2(\Omega)} + (F - cu, E \tilde{g})_{L^2(\Omega)}
\]
where $\tilde{g} \in H^{1/2}(\partial \Omega)$, and $E: H^{1/2}(\partial \Omega) \to H^1(\Omega)$ is any bounded extension operator. (Formally $\pn u := \nu_i a^{ij}\p_j u$, where $\nu:\partial \Omega \rightarrow \R^n$ is the unit outer normal to $\partial \Omega$). If $\Omega' \Subset \Omega$ is another Lipschitz domain, $\partial_{\nu} u|_{\partial \Omega'}$ is well defined and may be computed from $\Omega'$ or $\Omega \setminus \overline{\Omega}'$. Moreover,  
\begin{equation} \label{green_formula_weak}
(Lu, w)_{L^2(\Omega)} - (u, L^* w)_{L^2(\Omega)} = (\pn u, w)_{L^2(\partial \Omega)} - (u, \pn w)_{L^2(\partial \Omega)}
\end{equation}
whenever $w \in H^1(\Omega)$ satisfies $L^* w \in L^2(\Omega)$. Additionally, there is $p > 2$ such that if $g = 0$, then $\|\nabla u\|_{L^p(\Omega)} \leq C \|F\|_{L^2(\Omega)}$. The constants $C$ and $p$ only depend on $\Omega$, $n$, and $M$.
\end{lem}

\begin{proof}
The last statement follows from \cite[Theorem 1]{Meyers1963} (Lipschitz domains are admissible by \cite[Theorem 0.5]{JerisonKenig1995}). The other statements are standard (see for instance \cite{McLean}).
\end{proof}

We will next prove Theorem \ref{thm:CR}. Define the space 
\[
\widetilde{H}^{1/2}(\Gamma) = \text{closure of $\{ g \in H^{1/2}(\partial D_2) \,;\, \mathrm{supp}(g) \subset \Gamma \}$ in $H^{1/2}(\partial D_2)$}.
\]
Then $\widetilde{H}^{1/2}(\Gamma)$ is a closed subspace of $H^{1/2}(\partial D_2)$, and its dual space may be identified with $H^{-1/2}(\Gamma)$ (similarly as in \cite[Theorem 3.3]{CHM17}). Let $X$ be the closure of $S_1$ in $L^2(D_1)$, and consider the mapping
\begin{equation}
\label{eq:A}
\begin{split}
A: \widetilde{H}^{1/2}(\Gamma) &\rightarrow X \subset L^2(D_1),\\
g & \mapsto u|_{D_1},
\end{split}
\end{equation}
where $u \in S_2$ has boundary data $g$. The operator $A$ may be written as $Ag = Pg|_{D_1}$ where $P$ is the Poisson operator for $L$ in $D_2$. Its Banach space adjoint is given by the operator
\begin{equation}
\label{eq:A*}
\begin{split}
A': X \subset L^2(D_1) &\rightarrow  H^{-1/2}(\Gamma),\\
h & \mapsto \pn w|_{\Gamma},
\end{split}
\end{equation}
where $w$ and $h$ are related through 
\begin{align}
\label{eq:dualzero}
L^{\ast} w = \left\{
\begin{array}{ll}
h \mbox{ in } D_1,\\
0 \mbox{ in } D_2 \setminus \overline{D}_1,
\end{array}
\right. \ w = 0 \mbox{ on } \partial D_2.
\end{align}
In fact, the above notions are well defined by Lemma \ref{lemma_solvability_basic}, and \eqref{eq:A*} follows from the computation 
\begin{equation} \label{a_adjoint_computation}
(Ag, h)_{L^2(D_1)} = (u, L^* w)_{L^2(D_2)} = (g, \partial_{\nu} w)_{L^2(\partial D_2)}
\end{equation}
for $g \in \widetilde{H}^{1/2}(\Gamma)$ and $h \in L^2(D_1)$, where we have used \eqref{green_formula_weak}.

\begin{proof}[Proof of Theorem \ref{thm:CR}]
It is enough to show that the range of $A$ is a dense subspace in $X$. By the Hahn-Banach theorem, this will follow if we can show that any $h \in L^2(D_1)$ that satisfies 
\begin{equation} \label{a_orthogonality_condition}
(Ag, h)_{L^2(D_1)} = 0 \quad \text{for all $g \in \widetilde{H}^{1/2}(\Gamma)$}
\end{equation}
must also satisfy $(v, h)_{L^2(D_1)} = 0$ for all $v \in S_1$. But if $w$ is the solution of \eqref{eq:dualzero}, then \eqref{a_adjoint_computation} and  \eqref{a_orthogonality_condition} imply that $\partial_{\nu} w|_{\Gamma} = 0$. Thus $w$ solves 
\[
L^* w = 0 \mbox{ in } D_2 \setminus \overline{D}_1, \ \ w|_{\partial D_2} = \partial_{\nu} w|_{\Gamma} = 0.
\]
Since $D_2 \setminus \overline{D}_1$ was assumed connected, the unique continuation principle (for instance the version in \cite[Theorem 1.9]{ARV09}) implies that $w = 0$ in $D_2 \setminus \overline{D}_1$. It follows that 
\[
h = L^* (w|_{D_1}) \mbox{ where } w|_{\partial D_1} = \partial_{\nu} w|_{\partial D_1} = 0.
\]
Now if $v \in S_1$, we use \eqref{green_formula_weak} to conclude that 
\[
(v, h)_{L^2(D_1)} = (v, L^*(w|_{D_1}))_{L^2(D_1)} = (Lv, w|_{D_1})_{L^2(D_1)} = 0.
\]
Thus $h$ is $L^2$-orthogonal to $S_1$, which proves the theorem.
\end{proof}

\section{Quantitative Unique Continuation}
\label{sec:quant}

In the sequel, we argue that the validity of an approximation result as in Theorem \ref{thm:Runge_quant} is closely related to a quantitative unique continuation result. To this end, we first show that Theorem \ref{thm:Runge_quant} entails a quantitative unique continuation result. 

\begin{prop}
\label{prop:approx_qwuc}
Let $L$ be the operator from \eqref{eq:L} and let $D_1,D_2, \Gamma, S_1,S_2$ be as in Section \ref{sec:intro}.
Let $h \in S_1$ and define $w$ as a solution of the equation
\begin{align}
\label{eq:dual_1}
L^{\ast} w & = \left\{
\begin{array}{ll}
h \mbox{ in } D_1,\\
0 \mbox{ in } D_2 \setminus \overline{D}_1,
\end{array}
\right. \ 
w = 0 \mbox{ on } \partial D_2.
\end{align}
Assume that the result of Theorem \ref{thm:Runge_quant} holds. 
Then,
\begin{align}
\label{eq:unique_cont}
\|h\|_{L^2(D_1)} \leq C e^{C \epsilon^{-\mu}} \|\pn w\|_{H^{-1/2}(\Gamma)},
\end{align}
whenever $\epsilon > 0$ satisfies $\epsilon \|h\|_{H^{1}(D_1)} \leq \frac{1}{2} \|h\|_{L^2(D_1)}$.
\end{prop}

\begin{proof}
The argument for \eqref{eq:unique_cont} follows from Green's theorem \eqref{green_formula_weak}, which asserts that
\begin{align}
\label{eq:Green}
(u, h)_{L^2(D_1)} = (g, \p_{\nu} w)_{L^2(\partial D_2)},
\end{align}
where the functions $h$, $w$ are assumed to be related as in the formulation of the proposition, while the functions $u, g$ are connected through
\begin{align*}
L u &= 0 \mbox{ in } D_2,\ u = g \mbox{ on } \partial D_2.
\end{align*}
In particular, this implies that
\begin{align*}
\|h\|_{L^2(D_1)}^2 = (h-u,h)_{L^2(D_1)} + (g, \pn w)_{L^2(\partial D_2)}.
\end{align*}
Assuming that in the sense of \eqref{eq:approx} $u$ is an approximation to $h \in S_1$ for a given choice of $\epsilon>0$, where $g = u|_{\partial D_2} \in \widetilde{H}^{1/2}(\Gamma)$, we infer that
\begin{align*}
\|h\|_{L^2(D_1)}^2 &\leq \|g\|_{H^{1/2}(\partial D_2)} \|\pn w\|_{H^{-1/2}(\Gamma)}
 + \|u-h\|_{L^2(D_1)}\|h\|_{L^2(D_1)}\\
& \leq C e^{C \epsilon^{-\mu}} \|h\|_{L^2(D_1)} \|\pn w\|_{H^{-1/2}(\Gamma)} + \epsilon \|h\|_{H^{1}(D_1)}\|h\|_{L^2(D_1)}.
\end{align*}
Dividing by $\|h\|_{L^2(D_1)}$ and
choosing $\epsilon>0$ such that $\epsilon \|h\|_{H^{1}(D_1)} \leq \frac{1}{2} \|h\|_{L^2(D_1)}$ implies \eqref{eq:unique_cont}.
\end{proof}

\begin{rmk}
\label{rmk:contra}
Let us comment on this quantitative unique continuation result:
At first sight a result like this seems impossible, as for a general solution to \eqref{eq:dual_1} unique continuation from the boundary  clearly only implies that if $\pn w|_{\Gamma}$ vanishes, then $w=0$ in $D_2 \setminus \overline{D}_1$. In particular, in general it does not imply that $w=0$ in $D_1$.

However, in the present situation we have two strong additional assumptions: Firstly, $h$ itself is an element of $S_1$. Secondly, we assume that \eqref{eq:approx} holds. This corresponds to the assumption that $S_2|_{D_1}$ is dense in $S_1$. If $\pn w|_{\Gamma} =0$, \eqref{eq:Green} therefore yields
\begin{align}
\label{eq:orth}
(u,h)_{L^2(D_1)} = 0 \mbox{ for all } u \in S_2.
\end{align}
By density of $S_2|_{D_1}$ in $S_1$ this then implies that \eqref{eq:orth} also holds for all functions $u\in S_1$. Setting $u = h$, we therefore indeed obtain $h=0$ in accordance with \eqref{eq:unique_cont}. This resolves the initial ``paradox".
\end{rmk}

Next, we derive a quantitative unique continuation result. As in the qualitative proof of the Runge approximation property, we will exploit this in deducing the quantitative results of Theorem \ref{thm:Runge_quant} and \ref{thm:Runge_quant2}.

\begin{prop}
\label{prop:quant_unique}
Let $L$ be the operator from \eqref{eq:L} and let $D_1,D_2, \Gamma, S_1,S_2$ be as in Section \ref{sec:intro}.
Let $h \in S_1$ and define $w: D_2 \rightarrow \R$ as a solution to
\begin{align*}
L^{\ast} w = \left\{
\begin{array}{ll}
h \mbox{ in } D_1,\\
0 \mbox{ in } D_2 \setminus \overline{D}_1,
\end{array}
\right. \ w = 0 \mbox{ on } \partial D_2.
\end{align*}
Then there exist a parameter $\mu>0$ and a constant $C>1$ (depending on $D_1, D_2, \Gamma, n, K$) such that
\begin{equation}
\label{eq:unique_1a}
\begin{split}
\|w\|_{H^1(D_2 \setminus \overline{D}_1)} 
&\leq C  \frac{ \|h\|_{L^{2}(D_1)}  }{\log\left( C \frac{ \|h\|_{L^{2}(D_1)} }{\|\pn w\|_{H^{-1/2}(\Gamma)} }  \right)^{\mu}}.
\end{split}
\end{equation}
Moreover, if $G$ is a bounded Lipschitz domain with $G \Subset D_2 \setminus \overline{D}_1$, then 
\begin{equation}
\label{eq:unique_1b}
\begin{split}
\|w\|_{H^1(G)} 
&\leq C \|h\|_{L^{2}(D_1)} \left( \frac{\|\pn w\|_{H^{-1/2}(\Gamma)} }{ \|h\|_{L^{2}(D_1)} }  \right)^{\delta},
\end{split}
\end{equation}
where $C \geq 1$ and $\delta \in (0,1)$ depend on $D_1, D_2, G, \Gamma, n, K$.
\end{prop}

\begin{proof}
We write $\Omega = D_2 \setminus \overline{D}_1$. Since
\begin{align*}
L^{\ast} w = 0 \mbox{ in } \Omega, 
\end{align*}
the quantitative unique continuation result of \eqref{eq:unique_1a} follows for instance from \cite{ARV09}. Indeed, \cite[Theorem 1.9]{ARV09} implies that whenever $E \geq \|w\|_{H^1(\Omega)}$ and $\eta \geq \|\pn w\|_{H^{-1/2}(\Gamma)}$, one has 
\begin{equation} \label{wzero_initial_estimate}
\|w\|_{L^2(\Omega)} \leq (E+\eta) \omega \left( \frac{\eta}{E+\eta} \right),
\end{equation}
where $\omega(t) \leq C(\log \frac{1}{t})^{-\mu}$ for $0 \leq t \leq 1$ and where $C, \mu > 0$ only depend on the quantities in the statement of this proposition.

We wish to obtain an analogous statement for $\nabla w$. To do this, set 
\[
\Omega_r = \{ x \in \Omega : \dist(x, \partial \Omega) > r \}.
\]
Denoting the right hand side of \eqref{wzero_initial_estimate} by $B$, Caccioppoli's inequality yields 
\[
\| \nabla w \|_{L^2(\Omega_r)} \leq C r^{-1} \| w \|_{L^2(\Omega_{r/2})} \leq C r^{-1} B.
\]
On the other hand, for any $q > 1$ sufficiently close to $1$, and with $\frac{1}{q} + \frac{1}{q'} = 1$, we have 
\[
\| \nabla w \|_{L^2(\Omega \setminus \overline{\Omega}_r)} \leq \| \nabla w \|_{L^{2q}(D_2)} | \Omega \setminus \overline{\Omega}_r |^{\frac{1}{2q'}} \leq C \| \nabla w \|_{L^{2q}(D_2)} r^{\frac{1}{2q'}}.
\]
Choosing $q = p/2$ with $p$ as in Lemma \ref{lemma_solvability_basic} gives $\| \nabla w \|_{L^{2q}(D_2)} \leq C \| h \|_{L^2(D_1)}$. Combining the above facts leads to 
\[
\| \nabla w \|_{L^2(\Omega)} \leq C r^{-1} B + C \|h\|_{L^2(D_1)} r^{\frac{1}{2q'}}.
\]
Now choosing $r > 0$ so that both terms on the right are equal, we get 
\begin{equation} \label{nablawzero_estimate}
\| \nabla w \|_{L^2(\Omega)} \leq C B^{\alpha} \|h\|_{L^2(D_1)}^{1-\alpha}
\end{equation}
where $\alpha \in (0,1)$ only depends on $q$. Using the fact that $\|w\|_{H^1(D_2)} \leq C \|h\|_{L^2(D_1)}$, we may take $E = C \|h\|_{L^2(D_1)}$. Choosing also $\eta = \|\pn w\|_{H^{-1/2}(\Gamma)}$, we have $\eta \leq C E$ and thus \eqref{eq:unique_1a} follows by combining \eqref{wzero_initial_estimate} and \eqref{nablawzero_estimate}.

To prove \eqref{eq:unique_1b}, we invoke instead \cite[Theorem 1.7]{ARV09} which states that 
\[
\|w\|_{L^2(G)} \leq C (E+\eta)^{1-\delta} \eta^{\delta}
\]
for some $C \geq 1$ and $\delta \in (0,1)$. Arguing as above, we obtain a similar estimate for $\|\nabla w\|_{L^2(G)}$, which yields \eqref{eq:unique_1b} for some new $\delta \in (0,1)$.
\end{proof}

\section{Quantitative Runge Approximation}
\label{sec:quantrunge}

We next seek to show that the quantitative unique continuation estimate \eqref{eq:unique_1a} from Proposition \ref{prop:quant_unique} implies the approximation result \eqref{eq:approx}. To this end, recall from Section \ref{sec:qual} 
the space $X$ which is the closure of $S_1$ in $L^2(D_1)$, the mapping
\begin{equation}
\label{eq:Asecond}
\begin{split}
A: \widetilde{H}^{1/2}(\Gamma) &\rightarrow X \subset L^2(D_1),\\
g & \mapsto u|_{D_1},
\end{split}
\end{equation}
where $u \in S_2$ has boundary data $g$, and its Banach space adjoint $A': X \subset L^2(D_1) \rightarrow  H^{-1/2}(\Gamma)$, $h \mapsto \pn w|_{\Gamma}$,
where $w$ and $h$ are related through \eqref{eq:dualzero}.

From general functional analysis (see Remark 3.5 in \cite{RS17}), the Hilbert space adjoint $A^*$ of $A$ then becomes $A^{\ast}:=R A'$, where 
\begin{align}
\label{eq:Riesz}
R: H^{-1/2}(\Gamma) \rightarrow \widetilde{H}^{1/2}(\Gamma)
\end{align}
denotes the Riesz isomorphism between a Hilbert space and its dual space. In particular, we have that $\|A^{\ast} h\|_{H^{1/2}(\partial D_2)} = \|A' h\|_{H^{-1/2}(\Gamma)}$.

We seek to argue via a singular value decomposition as in \cite[Section 3]{RS17} or \cite{R95}.
To that end, we begin by studying the operator $A$ from \eqref{eq:Asecond}. 

\begin{lem}
\label{lem:singular_val_decom}
The operator $A$ in \eqref{eq:Asecond} is a compact, injective operator $\widetilde{H}^{1/2}(\Gamma) \to X$. Its range is dense. Further, there are orthonormal bases $\{\varphi_j\}_{j=1}^{\infty}$ of $\widetilde{H}^{1/2}(\Gamma)$ and $\{\psi_j\}_{j=1}^{\infty}$ of $X$, such that $A \varphi_j = \sigma_j \psi_j$ where $\sigma_j > 0$ are the singular values associated with the operator $A$.
\end{lem}
\begin{proof}
For compactness, if $(g_j)$ is a bounded sequence in $\widetilde{H}^{1/2}(\Gamma)$, then $\|A g_j\|_{H^1(D_1)} \leq C$. Thus a subsequence of $(A g_j)$ converges to some $h \in H^1(D_1)$ weakly in $H^1(D_1)$ and strongly in $L^2(D_1)$ by Rellich's theorem. By weak convergence, $h$ solves $L h = 0$ in $D_1$, showing that $h \in X$ and that $A$ is compact. The injectivity of $A$ follows from the unique continuation principle since $D_2\setminus \bar{D}_1$ is connected, and the qualitative Runge approximation (Theorem \ref{thm:CR} and its proof) shows that $A$ has dense range in $X$.

Therefore, the operator $A^{\ast}A:\widetilde{H}^{1/2}(\Gamma) \rightarrow \widetilde{H}^{1/2}(\Gamma)$ is a compact, self-adjoint, positive definite operator. The spectral theorem thus yields the existence of an orthonormal basis $\{\varphi_j\}_{j=1}^{\infty}$ of $\widetilde{H}^{1/2}(\Gamma)$ and a sequence of positive eigenvalues $\{\mu_j\}_{j=1}^{\infty}$ such that
\begin{align*}
A^{*} A \varphi_j = \mu_j \varphi_j.
\end{align*}
We define $\sigma_j = \mu_j^{1/2}$ and set $\psi_j = \sigma_j^{-1} A \varphi_j \in X$. We claim that $\{\psi_j\}_{j=1}^{\infty}$ is an orthonormal basis of $X$. As orthonormality follows by definition, it suffices to prove completeness of this set. This follows again from the qualitative Runge approximation result of Theorem \ref{thm:CR}. Indeed, if $v\in X$ is such that $(v,\psi_j)_{L^2(D_1)}=0$ for all $j$, then by density of $\varphi_j$ in $\widetilde{H}^{1/2}(\Gamma)$ this however entails that $(v,Ag)_{L^2(D_1)}=0$ for all $g\in \widetilde{H}^{1/2}(\Gamma)$. But since $A$ has dense range in $X$, this however implies that $(v, h)_{L^2(D_1)} = 0$ for all $h \in X$. Choosing $h = v$ implies $v = 0$, which concludes the completeness proof.
\end{proof}

With these preliminary results at hand, we proceed to the proof of Theorems \ref{thm:Runge_quant} and \ref{thm:Runge_quant2}.

\begin{proof}[Proof of Theorem \ref{thm:Runge_quant}]
Let $h \in S_1 \subset X$, and write $h = \sum\limits_{j=1}^{\infty} \beta_j \psi_j$. For $\alpha > 0$, we define
\begin{align*}
R_{\alpha} h := \sum\limits_{\sigma_j \geq \alpha} \frac{\beta_j}{\sigma_j} \varphi_j \in \widetilde{H}^{1/2}(\Gamma).
\end{align*} 
By orthonormality, we in particular obtain that
\begin{align}
\label{eq:boundary}
\|R_{\alpha} h \|_{H^{1/2}(\partial D_2)}^2 
= \sum\limits_{\sigma_j \geq \alpha} \frac{\beta_j^2}{\sigma_j^2} 
\leq \frac{1}{\alpha^2}\sum\limits_{\sigma_j \geq \alpha} \beta_j^2
\leq \frac{1}{\alpha^2}\|h\|_{L^2(D_1)}^2.
\end{align}
For the function $h\in S_1$ from above, we set $r_{\alpha}:=\sum\limits_{\sigma_j < \alpha} \beta_j \psi_j$ and define $w_{\alpha}$ as the solution of 
\begin{align*}
L^{\ast} w_{\alpha} = \left\{
\begin{array}{ll}
r_{\alpha} \mbox{ in } D_1,\\
0 \mbox{ in } D_2 \setminus \overline{D}_1,
\end{array}
\right. \ w_{\alpha} = 0 \mbox{ on } \partial D_2.
\end{align*}
By orthogonality and integration by parts using \eqref{green_formula_weak},  
we infer that 
\begin{equation*}
\begin{split}
\|A(R_{\alpha} h)-h\|_{L^2(D_1)}^2 
&= \|r_{\alpha}\|_{L^2(D_1)}^2 = (h,r_{\alpha})_{L^2(D_1)}
= (h, L^{\ast} w_{\alpha})_{L^2(D_1)}\\
& =(L h,  w_{\alpha})_{L^2(D_1)} - (\pn h, w_{\alpha})_{L^2(\partial D_1)} + (h, \pn w_{\alpha})_{L^2(\partial D_1)}.
\end{split}
\end{equation*}
Now $Lh = 0$ in $D_1$ and $L^* w_{\alpha} = 0$ in $D_2 \setminus \overline{D}_1$. Using trace estimates for solutions in the respective domains, combined with the quantitative unique continuation result of Proposition \ref{prop:quant_unique}, this leads to
\begin{equation*}
\begin{split} 
\|A(R_{\alpha} h)-h\|_{L^2(D_1)}^2 
&\leq \| \pn h \|_{H^{-1/2}(\partial D_1)} \| w_{\alpha} \|_{H^{1/2}(\partial D_1)} + \| h \|_{H^{1/2}(\partial D_1)} \| \pn w_{\alpha} \|_{H^{-1/2}(\partial D_1)} \\
&\leq C \| h\|_{H^{1}(D_1)}\|w_{\alpha}\|_{H^{1}(D_2 \setminus \overline{D}_1)}\\
&\leq  \| h\|_{H^{1}( D_1)} 
\frac{C \|r_{\alpha}\|_{L^2(D_1)}}{\log\left( C \frac{\|r_{\alpha}\|_{L^2(D_1)}}{\|\pn w_{\alpha}\|_{H^{-1/2}(\Gamma)}} \right)^{\mu}}.
\end{split}
\end{equation*}
Using the relation between $A^*$ and $A'$, the fact that $A^* \psi_j = \sigma_j \varphi_j$, and orthogonality, we have $\|A^* r_{\alpha}\|_{H^{1/2}(\partial D_2)} \leq \alpha \| r_{\alpha} \|_{L^2(D_1)}$ and 
\begin{align*}
\frac{1 }{\log\left( C \frac{\|r_{\alpha}\|_{L^2(D_1)}}{\|\pn w_{\alpha}\|_{H^{-1/2}(\Gamma)}} \right)}
=
\frac{1 }{-\log\left( C \frac{\|A^{\ast} r_{\alpha}\|_{H^{1/2}(\partial D_2)}}{\|r_{\alpha}\|_{L^2(D_1)}} \right)}
\leq \frac{1}{-\log\left(C \alpha \right)}.
\end{align*}
Defining $\alpha$ so that $\epsilon = C |\log\left( C \alpha \right)|^{-\mu}$ and dividing by $\|A(R_{\alpha} h)-h\|_{L^2(D_1)} = \|r_{\alpha}\|_{L^2(D_1)}$ concludes the proof.
\end{proof}

\begin{proof}[Proof of Theorem \ref{thm:Runge_quant2}]
Let $\tilde{h} \in H^1(\tilde{D})$ satisfy $L\tilde{h} = 0$ in $\tilde{D}$, and define $h := \tilde{h}|_{D_1}$. Then also $h \in S_1$, and we may define $R_{\alpha} h$, $r_{\alpha}$, and $w_{\alpha}$ as in the proof of Theorem \ref{thm:Runge_quant}. Now 
\begin{equation}
\begin{split}
\label{eq:error}
\|A(R_{\alpha} h)-h\|_{L^2(D_1)}^2 
&= \|r_{\alpha}\|_{L^2(D_1)}^2 = (h,r_{\alpha})_{L^2(D_1)}
= (\tilde{h}, L^{\ast} w_{\alpha})_{L^2(\tilde{D})}\\
& = - (\pn \tilde{h}, w_{\alpha})_{L^2(\partial \tilde{D})} + (\tilde{h}, \pn w_{\alpha})_{L^2(\partial \tilde{D})}.
\end{split}
\end{equation}
Define $G = U \setminus \overline{\tilde{D}}$ where $U$ is a Lipschitz domain with $\tilde{D} \Subset U \Subset D_2$ and $\partial G = \partial \tilde{D} \cup \partial U$ ($U$ can be obtained by enlarging $\tilde{D}$ slightly). Using that $L w_{\alpha} = 0$ in $G$, trace estimates together with \eqref{eq:unique_1b} yield 
\begin{equation*}
\begin{split} 
\|A(R_{\alpha} h)-h\|_{L^2(D_1)}^2 
&\leq \| \pn \tilde{h} \|_{H^{-1/2}(\partial \tilde{D})} \| w_{\alpha} \|_{H^{1/2}(\partial \tilde{D})} + \| \tilde{h} \|_{H^{1/2}(\partial \tilde{D})} \| \pn w_{\alpha} \|_{H^{-1/2}(\partial \tilde{D})} \\
&\leq C \| \tilde{h} \|_{H^{1}(\tilde{D})}\|w_{\alpha}\|_{H^{1}(G)}\\
&\leq  C \| \tilde{h} \|_{H^{1}(\tilde{D})} 
\|r_{\alpha}\|_{L^2(D_1)} \left( \frac{\|\pn w_{\alpha}\|_{H^{-1/2}(\Gamma)}}{\|r_{\alpha}\|_{L^2(D_1)}} \right)^{\delta}.
\end{split}
\end{equation*}
Repeating the argument in the proof of Theorem \ref{thm:Runge_quant} gives $\|A(R_{\alpha} h)-h\|_{L^2(D_1)} \leq C \| \tilde{h} \|_{H^{1}(\tilde{D}_1)} \alpha^{\delta}$. Since also $\|R_{\alpha} h \|_{H^{1/2}(\partial D_2)} \leq \alpha^{-1} \| h \|_{L^2(D_1)}$, choosing $\alpha = (\epsilon/C)^{1/\delta}$ finishes the proof.
\end{proof}

\section{Optimality}
\label{sec:optimal}

In order to infer optimality of the result of Theorem \ref{thm:Runge_quant} we consider the simplest possible case of harmonic functions. For these we easily obtain optimality of the bounds in \eqref{eq:approx}. The idea is to consider boundary values given by spherical harmonics on $\partial B_1$, so that the corresponding harmonic functions will decay rapidly toward the interior.

\begin{prop}
Let $L=\Delta$ and set $D_1=B_{1/2}$, $D_2=B_1$, and $\Gamma = \partial B_1$. There exists a sequence $(h_j) \subset S_1$ with $\|h_j\|_{H^1(D_1)}=1$ such that 
\begin{align*}
\mbox{for any $u \in S_2$ with }\|h_j - u\|_{L^2(D_1)}\leq (10 j)^{-1}, \mbox{ one has }
\|u\|_{H^{1/2}(\partial D_2)} \geq c e^{c j}.
\end{align*}
\end{prop}
\begin{proof}
Let $H_l$ be the subspace of $L^2(\partial B_1)$ consisting of spherical harmonics of degree $l$, associated with the eigenvalue $\lambda_l = l(l+n-2)$ of the spherical Laplacian. Working in polar coordinates $x=r \theta$ where $\theta \in \partial B_1$, any $u \in H^1(B_1)$ solving $\Delta u = 0$ in $B_1$ may be written as 
\[
u(r\theta) = \sum_{l=0}^{\infty} r^l u_l(\theta), \qquad u_l(\theta) = \sum_{m=1}^{N_l} c_{lm} \psi_{lm}(\theta) \in H_l
\]
where $\{ \psi_{l1}, \ldots, \psi_{lN_l} \}$ is an orthonormal basis of $H_l$. If $\tilde{u} = \sum r^l \tilde{u}_l(\theta)$ is another such function, then by orthogonality 
\begin{equation} \label{solid_spherical_harmonic_computation}
(u, \tilde{u})_{L^2(B_R)} = \sum_{l=0}^{\infty} \int_0^R r^{2l+n-1} (u_l, \tilde{u}_l)_{L^2(\partial B_1)} \,dr = \sum_{l=0}^{\infty} \frac{R^{2l+n}}{2l+n} (u_l, \tilde{u}_l)_{L^2(\partial B_1)}.
\end{equation}

We define $h_l$ in $B_1$ by 
\[
h_l(r\theta) := \alpha_l g_l(r\theta), \qquad g_l(r\theta) := r^l \psi_{l1}(\theta),
\]
where $\alpha_l := \| g_l \|_{H^1(D_1)}^{-1}$. Then $h_l$ is harmonic and $h_l|_{D_1} \in S_1$. By \eqref{solid_spherical_harmonic_computation}, 
\[
\| g_l \|_{L^2(D_1)}^2 = (2l+n)^{-1} 2^{-2l-n}
\]
and using integration by parts  
\[
\| \nabla g_l \|_{L^2(D_1)}^2 = ( \p_r g_l, g_l)_{L^2(\partial B_{1/2})} = l 2^{2-2l-n}.
\]
We obtain  
\begin{equation} \label{gl_norm_computations}
\| g_l \|_{L^2(D_1)} = (2l+n)^{-1/2} 2^{-l-n/2}, \ \  \| g_l \|_{H^1(D_1)} = (4l + (2l+n)^{-1})^{1/2} 2^{-l-n/2}.
\end{equation}

Let now $u \in S_2$. Choosing $c$ so that $c \alpha_l = \| g_l \|_{L^2(D_1)}^{-2} (u, g_l)_{L^2(D_1)}$, we may write 
\[
u = c \alpha_l g_l + v \text{ in } B_1, \qquad (v, g_l)_{L^2(D_1)} = 0.
\]
If additionally $\|u-h_l\|_{L^2(D_1)} \leq (10 l)^{-1}$, it follows that $|c-1| \alpha_l \| g_l \|_{L^2(D_1)} \leq (10 l)^{-1}$. Since $v$ also solves $\Delta v = 0$ in $B_1$, we may write $v = \sum r^k v_k(\theta)$ where $v_k \in H_k$, and the condition $(v, h_l)_{L^2(D_1)} = 0$ implies by \eqref{solid_spherical_harmonic_computation} that $v_l = \sum_{m=2}^{N_l} d_{lm} \psi_{lm}(\theta)$ (i.e.\ the coefficient for $\psi_{l1}$ is zero). Combining these facts shows that $u|_{\partial B_1}$ satisfies 
\[
u(\theta) = c \alpha_l \psi_{l1}(\theta) + w(\theta), \qquad (w, \psi_{l1})_{L^2(\partial B_1)} = 0.
\]
Therefore, since $|c \alpha_l - \alpha_l| \leq (10 l)^{-1} \| g_l \|_{L^2(D_1)}^{-1}$, the formulas \eqref{gl_norm_computations} yield that for $l$ large  
\[
\|u\|_{H^{1/2}(\partial B_1)} \geq (1 + \lambda_l^{1/2})^{1/2} |c \alpha_l| \geq l^{1/2} (\alpha_l - (10 l)^{-1} \| g_l \|_{L^2(D_1)}^{-1}) \geq c_n 2^{l/2}
\]
which implies the claimed lower bound. 
\end{proof}

\section[Application to Stability]{Application to the Calder\'on Problem with Local Data}
\label{sec:stability}

As an application of the quantified version of the Runge approximation, we demonstrate its applicability in inverse problems by providing a new stability proof for the partial data problem for the Schr\"odinger equation assuming that the potentials agree near the boundary (see \cite{AU04}). 
Although the result itself is not new (it had first been derived in \cite{F07} for the Schr\"odinger and in \cite{AK12} for the conductivity equation), we believe that the ideas which are used in our proof differ from the ones in \cite{F07}, \cite{AK12} (although not surprisingly quantitative unique continuation and propagation of smallness play a central role in all these results) and could be useful in other inverse problems. 

In order to state our main result, let $\Omega \subset \R^n$ be a bounded Lipschitz domain, let $q \in L^{\infty}(\Omega)$, and assume that $0$ is not a Dirichlet eigenvalue of $-\Delta+q$ in $\Omega$. Let $\Gamma$ be a nonempty open subset of $\partial \Omega$. We consider the associated \emph{local Dirichlet-to-Neumann map}
\begin{align*}
\Lambda_{q}^{\Gamma}: \widetilde{H}^{1/2}(\Gamma) \rightarrow H^{-1/2}(\Gamma), \ g \mapsto \partial_{\nu} u|_{\Gamma},
\end{align*}
where $u \in H^1(\Omega)$ is the solution of the problem 
\begin{align*}
(-\D+q)u & = 0 \mbox{ in } \Omega,\\
u & = g \mbox{ on } \partial \Omega.
\end{align*}
We also write 
\[
\| \Lambda_{q}^{\Gamma} \|_* = \| \Lambda_{q}^{\Gamma} \|_{\widetilde{H}^{1/2}(\Gamma) \rightarrow H^{-1/2}(\Gamma)}.
\]

With this notation at hand, we will prove the following stability estimate. The proof is based on quantifying the argument of \cite{AU04} by using Theorem \ref{thm:Runge_quant2}.

\begin{prop}[Stability]
\label{prop:stability}
Let $n \geq 3$, let $\Omega \subset \R^n$ be bounded a Lipschitz domain, and suppose that
$q_1,q_2 \in L^{\infty}(\Omega)$ are such that zero is not a Dirichlet eigenvalue of $-\D + q_j$ in $\Omega$ and 
\begin{align*}
\|q_j\|_{L^{\infty}(\Omega)} \leq M <\infty \mbox{ for } j =1,2.
\end{align*}
Assume further that $q_1=q_2$ in $\Omega \setminus \overline{\Omega}'$ where $\Omega' \Subset \Omega$ is another bounded Lipschitz domain such that $\Omega \setminus \overline{\Omega}'$ is connected. Then 
\begin{align*}
\|q_1-q_2\|_{H^{-1}(\Omega)} \leq \omega(\|\Lambda_{q_1}^{\Gamma}-\Lambda_{q_2}^{\Gamma}\|_{\ast})
\end{align*}
where $\omega(t)= C|\log(t)|^{-\sigma}$ for $t \in [0,1]$, and $C > 1, \sigma > 0$ only depend on $\Omega, \Omega', \Gamma, n, M$.
\end{prop}

\begin{proof} 
We will construct complex geometrical optics solutions $u_j$ solving $(-\Delta+q_j) u_j = 0$ in $\Omega$ (see \cite{SU87}). Fix $k \in \R^n$, choose unit vectors $l, m \in \R^n$ with $k \cdot l = k \cdot m = l \cdot m = 0$ (here we use the assumption $n \geq 3$), and for $\tau \geq |k|/2$ define the complex vectors 
\begin{align*}
\rho_1 = \tau m + i \left( -\frac{k}{2} + \sqrt{\tau^2 - \frac{|k|^2}{4}} \,l \right), \quad \rho_2 = -\tau m + i \left( -\frac{k}{2} - \sqrt{\tau^2 - \frac{|k|^2}{4}} \,l \right).
\end{align*}
By \cite{SU87}, if $\tau \geq 1$ is large enough, there exist $u_j \in H^1(\Omega)$ solving $(-\Delta+q_j) u_j = 0$ in $\Omega$ and having the form 
\begin{align*}
u_j = e^{x \cdot \rho_j}(1+ \psi_{i}(x,\rho_j)), \quad j\in\{1,2\},
\end{align*}
where 
\begin{align*}
\|\psi_j\|_{L^2(\Omega)} \leq \frac{C}{\tau} \|q_j\|_{L^2(\Omega)}, \quad \|\psi_j\|_{H^1(\Omega)} \leq C \|q_j\|_{L^2(\Omega)}.
\end{align*}

Given an error threshold $\epsilon>0$, whose precise value will be fixed later, we use Theorem \ref{thm:Runge_quant2} (with $D_1 = \Omega'$, $\tilde{D}$ being a slight fattening of $\Omega'$ and $D_2 = \Omega$) to find $\tilde{u}_j \in H^1(\Omega)$, $j\in\{1,2\}$, solving $(-\Delta+q_j) \tilde{u}_j = 0$ in $\Omega$ with $\tilde{u}_j|_{\partial \Omega \setminus \overline{\Gamma}} = 0$, such that one has 
\[
\| \tilde{u}_j - u_j \|_{L^2(\Omega')} \leq \epsilon \| u_j \|_{H^1(\tilde{D})}, \qquad \| \tilde{u}_j \|_{H^{1/2}(\partial \Omega)} \leq C \epsilon^{-\mu} \| u_j \|_{L^2(\Omega')}.
\]
Since $q_1 = q_2$ in $\Omega \setminus \overline{\Omega}'$, integration by parts using \eqref{green_formula_weak} gives that 
\begin{align*}
\int\limits_{\Omega'}(q_1-q_2)\tilde{u}_1 \tilde{u}_2 \,dx &= \int\limits_{\Omega}(q_1-q_2)\tilde{u}_1 \tilde{u}_2 \,dx = (\Delta \tilde{u}_1, \tilde{u}_2)_{L^2(\Omega)} - (\tilde{u}_1, \Delta \tilde{u}_2)_{L^2(\Omega)} \\
&= (\pn \tilde{u}_1, \tilde{u}_2)_{L^2(\p \Omega)} - (\tilde{u}_1, \pn \tilde{u}_2)_{L^2(\p \Omega)} = (\Lambda_{q_1}^{\Gamma} \tilde{u}_1, \tilde{u}_2)_{L^2(\p \Omega)} - (\tilde{u}_1, \Lambda_{q_2}^{\Gamma} \tilde{u}_2)_{L^2(\p \Omega)} \\
 &= ( (\Lambda_{q_1}^{\Gamma} - \Lambda_{q_2}^{\Gamma}) \tilde{u}_1, \tilde{u}_2 )_{L^2(\partial \Omega)}.
\end{align*}
Here we used that $(\Lambda_q^{\Gamma} g, h)_{L^2(\partial \Omega)} = (g, \Lambda_q^{\Gamma} h)_{L^2(\partial \Omega)}$ (this follows from \eqref{green_formula_weak}).
Hence,
\begin{align*}
\int\limits_{\Omega'}(q_1-q_2) e^{-i k \cdot x} \,dx
& = -\int\limits_{\Omega'}(q_1-q_2)e^{-ik \cdot x}(\psi_1+\psi_2 + \psi_1 \psi_2) \,dx
+ \int\limits_{\Omega'}(q_1-q_2)(u_1-\tilde{u}_1) u_2 \,dx \\
& \quad + \int\limits_{\Omega'}(q_1-q_2)(u_2-\tilde{u}_2) \tilde{u}_1 \,dx
+ ( (\Lambda_{q_1}^{\Gamma} - \Lambda_{q_2}^{\Gamma}) \tilde{u}_1, \tilde{u}_2 )_{L^2(\partial \Omega)}. 
\end{align*}
As a consequence, if $q_j$ are extended by zero to $\R^n$, 
\begin{align*}
|\mathcal{F}(q_1-q_2)(k)| 
&\leq C(\Omega)M ( \|\psi_1\|_{L^2(\Omega)}+ \|\psi_2\|_{L^2(\Omega)} + \|\psi_1\|_{L^2(\Omega)}\,\|\psi_2\|_{L^2(\Omega)}) \\
& \quad + 2M \|u_2\|_{L^2(\Omega)}\|u_1-\tilde{u}_1\|_{L^2(\Omega')}
+ 2M \|\tilde{u}_1\|_{L^2(\Omega')}\|u_2-\tilde{u}_2\|_{L^2(\Omega')}\\
& \quad + \|\Lambda_{q_1}^{\Gamma}-\Lambda_{q_2}^{\Gamma}\|_{*}\|\tilde{u}_1\|_{H^{1/2}(\partial \Omega)} \|\tilde{u}_2\|_{H^{1/2}(\partial \Omega)}\\
&\leq \frac{C}{\tau}  + C \epsilon (\|u_2\|_{L^2(\Omega)}\|u_1\|_{H^1(\Omega)} + (\|u_1\|_{L^2(\Omega)} + \epsilon \|u_1\|_{H^1(\Omega)}) \|u_2\|_{H^1(\Omega)} ) \\
& \quad + \|\Lambda_{q_1}^{\Gamma}-\Lambda_{q_2}^{\Gamma}\|_{*} C  \epsilon^{-2\mu}\|u_1\|_{L^{2}( \Omega)} \|u_2\|_{L^{2}(\Omega)}\\
&\leq \frac{C}{\tau}  + C \epsilon \tau^2 e^{C\tau} + C  \epsilon^{-2\mu} e^{C\tau} \|\Lambda_{q_1}^{\Gamma}-\Lambda_{q_2}^{\Gamma}\|_{*}
\end{align*}
where $C$ and $\mu$ depend on $\Omega, \Omega', \Gamma, n, M$. Therefore, for any $\rho \geq 1$, 
\begin{align*}
\|q_1-q_2\|_{H^{-1}(\Omega)}^2
& \leq \int\limits_{|k|<\rho}|\mathcal{F}(q_1-q_2)(k)|^2(1+|k|^2)^{-1} \,dk +  \int\limits_{|k|\geq \rho}|\mathcal{F}(q_1-q_2)(k)|^2(1+\rho^2)^{-1} \,dk\\
& \leq C \rho^{n-2} \left( \frac{1}{\tau^2}  + \epsilon^2 e^{C \tau} +\|\Lambda_{q_1}^{\Gamma}-\Lambda_{q_2}^{\Gamma}\|_{*}^2  \epsilon^{-4\mu} e^{C \tau} \right) + \frac{1}{1+\rho^2}\|q_1-q_2\|_{L^2(\Omega)}^2\\
& \leq C \rho^{n-2} \left( \frac{1}{\tau^2}  + \epsilon^2 e^{C \tau} +\|\Lambda_{q_1}^{\Gamma}-\Lambda_{q_2}^{\Gamma}\|_{*}^2  \epsilon^{-4\mu} e^{C \tau} \right) + \frac{C}{\rho^2}.
\end{align*}
Choosing $\rho = \tau^{\frac{2}{n}}$, $\epsilon = \|\Lambda_{q_1}^{\Gamma}-\Lambda_{q_2}^{\Gamma}\|_{*}^{\frac{2}{2+4\mu}}$ and $\tau = C^{-1} |\log(\|\Lambda_{q_1}^{\Gamma}-\Lambda_{q_2}^{\Gamma}\|_{*}^{\frac{2}{2+4\mu}})|$ then implies the desired result.
\end{proof}

\begin{rmk}
Alternatively, one could use Theorem \ref{thm:Runge_quant2} to give an easy proof of an analogue of \cite[Theorem 3.1]{AK12}: if $\tilde{D}$ is a suitable Lipschitz domain with $\Omega' \Subset \tilde{D} \Subset \Omega$ in the setup of Proposition \ref{prop:stability}, and if $\Lambda_{q}^{\partial \tilde{D}}$ is the Dirichlet-to-Neumann map on the full boundary of $\tilde{D}$, then 
\[
\| \Lambda_{q_1}^{\partial \tilde{D}} - \Lambda_{q_2}^{\partial \tilde{D}} \|_{H^{1/2}(\partial \tilde{D}) \to H^{-1/2}(\partial \tilde{D})} \leq C \|\Lambda_{q_1}^{\Gamma}-\Lambda_{q_2}^{\Gamma}\|_{\ast}^{\beta}
\]
for some $\beta \in (0,1)$. Combining this with \cite{A88} would imply Proposition \ref{prop:stability}.
\end{rmk}

\bibliographystyle{alpha}
\bibliography{citations_Runge}

\begin{thebibliography}{CWHM17}

\bibitem[AK12]{AK12}
Giovanni Alessandrini and Kyoungsun Kim.
\newblock Single-logarithmic stability for the {C}alder\'on problem with local
  data.
\newblock {\em J. Inverse Ill-Posed Probl.}, 20(4):389--400, 2012.

\bibitem[Ale88]{A88}
Giovanni Alessandrini.
\newblock Stable determination of conductivity by boundary measurements.
\newblock {\em Appl. Anal.}, 27(1-3):153--172, 1988.

\bibitem[ARRV09]{ARV09}
Giovanni Alessandrini, Luca Rondi, Edi Rosset, and Sergio Vessella.
\newblock The stability for the {C}auchy problem for elliptic equations.
\newblock {\em Inverse Problems}, 25(12):123004, 47, 2009.

\bibitem[AU04]{AU04}
Habib Ammari and Gunther Uhlmann.
\newblock Reconstruction of the potential from partial {C}auchy data for the
  {S}chr\"odinger equation.
\newblock {\em Indiana Univ. Math. J.}, 53(1):169--183, 2004.

\bibitem[Bro62a]{B62a}
Felix~E. Browder.
\newblock Functional analysis and partial differential equations. {II}.
\newblock {\em Math. Ann.}, 145:81--226, 1961/1962.

\bibitem[Bro62b]{B62}
Felix~E Browder.
\newblock Approximation by solutions of partial differential equations.
\newblock {\em Amer. J. Math.}, 84(1):134--160, 1962.

\bibitem[CWHM17]{CHM17}
S.~N. Chandler-Wilde, D.~P. Hewett, and A.~Moiola.
\newblock Sobolev spaces on non-{L}ipschitz subsets of {$\Bbb{R}^n$} with
  application to boundary integral equations on fractal screens.
\newblock {\em Integral Equations Operator Theory}, 87(2):179--224, 2017.

\bibitem[Fat07]{F07}
Ines~Kamoun Fathallah.
\newblock Stability for the inverse potential problem by the local
  {D}irichlet-to-{N}eumann map for the {S}chr{\"o}dinger equation.
\newblock {\em Appl. Anal.}, 86(7):899--914, 2007.

\bibitem[Geb08]{G08}
Bastian Gebauer.
\newblock Localized potentials in electrical impedance tomography.
\newblock {\em Inverse Probl. Imaging}, 2(2):251--269, 2008.

\bibitem[Ike98]{Ikehata1998}
Masaru Ikehata.
\newblock Reconstruction of the shape of the inclusion by boundary
  measurements.
\newblock {\em Comm. Partial Differential Equations}, 23(7-8):1459--1474, 1998.

\bibitem[Ike13]{Ikehata2013}
Masaru Ikehata.
\newblock Analytical methods for extracting discontinuity in inverse problems:
  the probe method after 10 years.
\newblock {\em Sugaku Expositions}, 26(1):1--28, 2013.

\bibitem[Isa88]{I88}
Victor Isakov.
\newblock On uniqueness of recovery of a discontinuous conductivity
  coefficient.
\newblock {\em Comm. Pure Appl. Math.}, 41(7):865--877, 1988.

\bibitem[JK95]{JerisonKenig1995}
David Jerison and Carlos~E. Kenig.
\newblock The inhomogeneous {D}irichlet problem in {L}ipschitz domains.
\newblock {\em J. Funct. Anal.}, 130(1):161--219, 1995.

\bibitem[KV85]{KV85}
R.~V. Kohn and M.~Vogelius.
\newblock Determining conductivity by boundary measurements. {II}. {I}nterior
  results.
\newblock {\em Comm. Pure Appl. Math.}, 38(5):643--667, 1985.

\bibitem[Lax56]{L56}
Peter~D Lax.
\newblock A stability theorem for solutions of abstract differential equations,
  and its application to the study of the local behavior of solutions of
  elliptic equations.
\newblock {\em Comm. Pure Appl. Math.}, 9(4):747--766, 1956.

\bibitem[Mal56]{M54}
Bernard Malgrange.
\newblock Existence et approximation des solutions des \'equations aux
  d\'eriv\'ees partielles et des \'equations de convolution.
\newblock {\em Ann. Inst. Fourier, Grenoble}, 6:271--355, 1955--1956.

\bibitem[McL00]{McLean}
William McLean.
\newblock {\em Strongly elliptic systems and boundary integral equations}.
\newblock Cambridge University Press, Cambridge, 2000.

\bibitem[Mey63]{Meyers1963}
Norman~G. Meyers.
\newblock An {$L^{p}$}-estimate for the gradient of solutions of second order
  elliptic divergence equations.
\newblock {\em Ann. Scuola Norm. Sup. Pisa (3)}, 17:189--206, 1963.

\bibitem[NUW05]{NUW05}
Gen Nakamura, Gunther Uhlmann, and Jenn-Nan Wang.
\newblock Oscillating-decaying solutions, {R}unge approximation property for
  the anisotropic elasticity system and their applications to inverse problems.
\newblock {\em J. Math. Pures Appl. (9)}, 84(1):21--54, 2005.

\bibitem[Rob95]{R95}
Luc Robbiano.
\newblock Fonction de co{\^u}t et contr{\^o}le des solutions des {\'e}quations
  hyperboliques.
\newblock {\em Asymptotic Analysis}, 10(2):95--115, 1995.

\bibitem[RS17]{RS17}
Angkana R{\"u}land and Mikko Salo.
\newblock The fractional {C}alder{\'o}n problem: Low regularity and stability.
\newblock {\em ArXiv preprint, August}, 2017.

\bibitem[SU87]{SU87}
John Sylvester and Gunther Uhlmann.
\newblock A global uniqueness theorem for an inverse boundary value problem.
\newblock {\em Ann. of Math.}, 125(1):153--169, 1987.

\end{thebibliography}

\end{document}